\documentclass{amsart}
\usepackage{amssymb}
\usepackage{amsfonts}
\usepackage{amssymb}
\usepackage{amsmath}
\usepackage{amsthm}
\usepackage{enumerate}
\usepackage{url}
\usepackage{tabularx}
\usepackage{centernot}
\usepackage{mathtools}
\usepackage{stmaryrd}
\usepackage{amsthm,amssymb}
\usepackage{etoolbox}
\usepackage{tikz}
\usepackage{amssymb}
\usetikzlibrary{matrix}
\usepackage{tikz-cd}
\usepackage{tikz}
\usepackage{marginnote}
\definecolor{mygray}{gray}{0.85}
\usepackage[backgroundcolor=mygray,colorinlistoftodos,prependcaption,textsize=small]{todonotes}
\usepackage{xargs}                      

\renewcommand{\leq}{\leqslant}
\renewcommand{\geq}{\geqslant}

\makeatletter
\def\subsection{\@startsection{subsection}{3}%
  \z@{.5\linespacing\@plus.7\linespacing}{.3\linespacing}%
  {\bfseries\centering}}
\makeatother

\makeatletter
\def\subsubsection{\@startsection{subsubsection}{3}%
  \z@{.5\linespacing\@plus.7\linespacing}{.3\linespacing}%
  {\centering}}
\makeatother

\makeatletter
\def\myfnt{\ifx\protect\@typeset@protect\expandafter\footnote\else\expandafter\@gobble\fi}
\makeatother

\newtheorem{theorem}{Theorem}[section]

\newtheorem{definition}[theorem]{Definition}
\newtheorem{lemma}[theorem]{Lemma}
\newtheorem{proposition}[theorem]{Proposition}

\newtheorem{problem}[theorem]{Problem}

\newtheorem{fact}[theorem]{Fact}

\newtheorem{remark}[theorem]{Remark}

\newtheorem{notation}[theorem]{Notation}

\newtheorem{conjecture}[theorem]{Conjecture}

\newtheorem*{theorem1}{Theorem~1.1}
\newtheorem*{proposition2}{Proposition~1.2}
\newtheorem*{theorem6}{Theorem~1.6}
\newtheorem*{proposition8}{Proposition~1.8}
\newtheorem*{lemma9}{Lemma~1.9}
\newtheorem*{conjecture10}{Conjecture~1.10}

\newcounter{claimcounter}
\numberwithin{claimcounter}{theorem}

\begin{document}

	\begin{abstract} We prove that no quantifier-free formula in the language of group theory can define the $\aleph_1$-half graph in a Polish group, thus generalising some results from \cite{paolini&shelah_metric_cyclic}. We then pose some questions on the space of groups of automorphisms of a given Borel complete class, and observe that this space must contain at least one uncountable group. 
Finally, we prove some results on the structure of the group of automorphisms of a locally finite group: firstly, we prove that it is not the case that every group of automorphisms of a graph of power $\lambda$ is the group of automorphism of a locally finite group of power $\lambda$; secondly, we conjecture that the group of automorphisms of a locally finite group of power $\lambda$ has a locally finite subgroup of power $\lambda$, and reduce the problem to a problem on $p$-groups, thus settling the \mbox{conjecture in the case $\lambda = \aleph_0$.}
\end{abstract}

\title{Some Results on Polish Groups}
\thanks{Partially supported by European Research Council grant 338821. No. 1155 on Shelah's publication list. The first author would like to thank Riccardo Camerlo for useful discussions related to the questions posed in Section~3.}

\author{Gianluca Paolini}
\address{Einstein Institute of Mathematics,  The Hebrew University of Jerusalem, Israel}

\author{Saharon Shelah}
\address{Einstein Institute of Mathematics,  The Hebrew University of Jerusalem, Israel \and Department of Mathematics,  Rutgers University, U.S.A.}

\date{\today}
\maketitle

\section{Introduction}

	We collect some results (of different nature) on the theory of Polish groups.

\smallskip 

\noindent \underline{\em  Section 2. Definable $\aleph_1$-Half Graphs in Polish Groups}.

\smallskip 
 
\noindent By the $\aleph_1$-half graph $\Gamma(\aleph_1)$ we mean the graph on vertex set $\{ a_{\alpha} : \alpha < \aleph_1 \} \cup  \{ b_{\beta}: \beta < \aleph_1 \}$ with edge relation $a_\alpha E_{\Gamma} b_\beta$ if and only if  $\alpha < \beta$. In the process of characterization of the graph products of cyclic groups embeddable in a Polish group \cite{paolini&shelah_metric_cyclic}, we observed that the commutation relation $x^{-1}y^{-1}xy = e$ can never define the $\aleph_1$-half graph in a Polish group $G$. Here we generalize this to:
	
	\begin{theorem}\label{theorem_half_1} No quantifier-free formula $\varphi(\bar{x}, \bar{y})$ in the language of group theory can define the $\aleph_1$-half graph in a Polish group $G$.
\end{theorem}

\noindent We actually prove a stronger result of independent interest, i.e. that Polish groups do not admit ``polarized $\aleph_1$-partitions'', see Theorem~\ref{stability_claim} for the detailed statement of this result. Finally, we would like to mention that Theorem~\ref{theorem_half_1} can be considered as a form of model-theoretic stability for Polish groups, in this direction see also~\cite{Sh849}.

\medskip 

\noindent \underline{\em  Section 3. The Space of Automorphism Groups of a Borel Complete Class}.

\smallskip 

\noindent By a Borel complete class we mean a Borel class $\mathbf{K}$ of structures with domain $\omega$ in a fixed language $L$ such that the isomorphism relation on $\mathbf{K}$ is as complicated as possible (equivalently, the countable graph isomorphism relation is reducible to it -- cf. Definition \ref{borel_complete}). We wonder here: how complex can $Aut(\mathbf{K}) = \{ Aut(A) : A \in \mathbf{K} \}$ be for a given Borel complete class? Can $Aut(\mathbf{K})$ contain only one isomorphism type, resp. finitely many, resp. countably many (cf. Problem~\ref{problem_sec2})? In this direction:  

	\begin{proposition}\label{theorem2} Let $\mathbf{K}$ be a Borel class of $L$-structures with domain $\aleph_0$ such that for every $G \in Aut(\mathbf{K})$ we have that $|G| \leq \aleph_0$.  Then the isomorphism relation on $\mathbf{K}$ {\em is} Borel, and so in particular $\mathbf{K}$ is {\em not} Borel complete (cf. Definition \ref{borel_complete}).
\end{proposition}
	
	On questions affine to this topic see also the interesting recent work~\cite{kech_tent}.

\medskip

\noindent \underline{\em  Section 4. Group of Automorphisms of Locally Finite Groups}. 

\begin{notation}\label{notation_locally_finite}
	\begin{enumerate}[(1)]
	\item We denote by $\mathbf{K}_{\mathrm{lf}}$ the class of locally finite groups.
	\item We denote by $\mathbf{K}_{\mathrm{gf}}$ the class of graphs.
	\item For $\mathbf{K}$ a class of $L$-structures and $\lambda$ an infinite cardinal, we let:
	\begin{enumerate}[(3.1)]
	\item $\mathbf{K}^{\lambda} = \{ M \in \mathbf{K} : \text{the domain of $M$ is } \lambda \}$;
	\item $Aut(\mathbf{K}) = \{ Aut(M) \; : M \in \mathbf{K} \}$.
\end{enumerate}	
\end{enumerate}
\end{notation}

	\begin{fact}[\cite{mekler}] The class $\mathbf{K}^{\aleph_0}_{\mathrm{lf}}$ is Borel complete (cf. Definition~\ref{borel_complete}). 
\end{fact}

	In this section we deal with the following problem:
	
	\begin{problem}\label{char_autos} Characterize $Aut(\mathbf{K}^{\lambda}_{\mathrm{lf}})$, for $\lambda \geq \aleph_0$.
\end{problem}

	In this direction we first prove:
	
	\begin{theorem}\label{auto_graphs} Let $\lambda \geq \aleph_0$, then:
	 $$Aut(\mathbf{K}^{\lambda}_{\mathrm{gf}}) \neq Aut(\mathbf{K}^{\lambda}_{\mathrm{lf}}).$$
\end{theorem}

	The proof of Theorem~\ref{auto_graphs} leads to the following conjecture:
	
	\begin{conjecture}\label{conjecture_on_auto_intro} If $G \in \mathbf{K}^{\lambda}_{\mathrm{lf}}$, then $Aut(G)$ has a locally finite subgroup of power~$\lambda$.
\end{conjecture}

	In the case of $\aleph_0$ we prove that this is indeed the case:
	
	\begin{proposition}\label{lemma_for_countable} If $G \in \mathbf{K}^{\aleph_0}_{\mathrm{lf}}$, then $Aut(G)$ has a locally finite infinite subgroup.
\end{proposition}

	On the other hand, we do not settle here Conjecture~\ref{conjecture_on_auto_intro} in general, but we prove:
	
	\begin{lemma}\label{reduction_lemma} To prove Conjecture~\ref{conjecture_on_auto_intro} it suffices to prove Conjecture~\ref{conjecture_simplified}, where:
\end{lemma}

	\begin{conjecture}\label{conjecture_simplified} If $\lambda > \aleph_0$, $G \in \mathbf{K}^{\lambda}_{\mathrm{lf}}$ is an abelian $p$-group, $H \leq G$ and $|H| < \lambda$, then $Aut_H(G)$ has a locally finite subgroup of power~$\lambda$.
\end{conjecture}

	Finally, we would like to mention that in \cite{hall_gp} we give a close analysis of the group of automorphisms of Philip Hall's universal locally finite group.

\section{Definable $\aleph_1$-Half Graphs in Polish Groups}


	\begin{theorem}\label{stability_claim} Let $G$ be a Polish group, and for $\ell < \ell(*) < \omega$ let:
\begin{enumerate}[(i)]
	\item $\bar{\mathbf{g}}^\ell = (\bar{g}^\ell_{\alpha} : \alpha \in A_\ell)$;
	\item $A_{\ell} \in [\omega_1]^{\aleph_1}$;
	\item $\bar{g}^\ell_{\alpha} \in G^{n(\ell)}$;
	\item $\Delta$ be a finite set of q.f. formulas of the form $\varphi(\bar{x}^0_{n(0)}, ..., \bar{x}^{\ell(*)-1}_{n(\ell(*)-1)})$ in the language of group theory such that $lg(\bar{x}^\ell_{n(\ell)}) = n(\ell)$.
\end{enumerate}
Then there are $B_{\ell} \in [A_{\ell}]^{\aleph_1}$, for $\ell < \ell(*)$, and truth value $t \in \{0, 1\}$ such that if $\alpha(\ell) \in B_\ell$,  for $\ell < \ell(*)$, then:
$$G \models \varphi^t(\bar{g}^0_{\alpha(0)}, ..., \bar{x}^{\ell(*)-1}_{\alpha(\ell(*)-1)}).$$
\end{theorem}

	\begin{proof} First of all notice that it suffices to prove the claim for:
	$$\Delta = \{ \sigma(\bar{x}^0_{n(0)}, ..., \bar{x}^{\ell(*)-1}_{n(\ell(*)-1)}) = e \},$$
and $\sigma(\bar{x}^0_{n(0)}, ..., \bar{x}^{\ell(*)-1}_{n(\ell(*)-1)})$ a term in the language of group theory $L = \{ e, \cdot, ()^{-1}\}$.
\newline [Why? First of all, without loss of generality, we can assume that each $\varphi \in \Delta$ is a Boolean combination of formulas of the form $\sigma(\bar{x}^0_{n(0)}, ..., \bar{x}^{\ell(*)-1}_{n(\ell(*)-1)}) = e$. So let $(\sigma_i(\bar{x}^0_{n(0)}, ..., \bar{x}^{\ell(*)-1}_{n(\ell(*)-1)}) = e : i < i(*) < \omega)$ list them. Now choose $(B_{i, \ell} : \ell < \ell(*))$ by induction on $i \leq i*$ such that:
\begin{enumerate}[(a)]
	\item $B_{i, \ell} \in [\omega_1]^{\aleph_1}$;
	\item $B_{0, \ell} = A_{\ell}$;
	\item $B_{i+1, \ell} \subseteq B_{i, \ell}$;
	\item $(B_{i+1, \ell} : \ell < \ell(*))$ satisfies the desired conclusion for:
	$$\Delta = \{ \sigma_i(\bar{x}^0_{n(0)}, ..., \bar{x}^{\ell(*)-1}_{n(\ell(*)-1)}) = e  \}.$$
\end{enumerate}
Then $(B_{i(*), \ell} : \ell < \ell(*))$ is as wanted.]
\newline Let $(G, d)$ witness the Polishness of $G$. For $\ell < \ell(*)$ and $\alpha, \beta \in A_{\ell}$, let $d(\bar{g}^\ell_{\alpha}, \bar{g}^\ell_{\beta}) = max\{ d(g^{\ell, i}_{\alpha}, g^{\ell, i}_{\beta}) : i < n(\ell)\}$, and:
	$$\mathcal{U}_\ell = \{ \alpha < \omega_1 : \text{ for some } \varepsilon \in (0, 1)_{\mathbb{R}} \text{ the set } \{ \beta \in A_\ell : d(\bar{g}^\ell_{\alpha}, \bar{g}^\ell_{\beta}) < \varepsilon \} \text{ is countable}\}.$$
Since $(G, d)$ is separable, for every $\ell < \ell(*)$, the set $\mathcal{U}_{\ell}$ is countable, and so we can find $\alpha(*) < \omega_1$ such that $\bigcup_{\ell < \ell(*)} \mathcal{U}_\ell \subseteq \alpha(*)$. Now, if $B_{\ell} = A_\ell - \alpha(*)$ is such that for every $\alpha({\ell}) \in B_{\ell}$ we have that $G \models \sigma(\bar{g}^0_{\alpha(0)}, ..., \bar{g}^{\ell(*)-1}_{\alpha(\ell(*)-1)}) = e$, then we are done. So suppose that this is not the case, then we can find $\alpha_{\ell} \in A_{\ell} - \alpha(*)$ such that:
$$\varepsilon = d(\sigma(\bar{g}^0_{\alpha(0)}, ..., \bar{g}^{\ell(*)-1}_{\alpha(\ell(*)-1)}), e) \neq 0.$$
As $G$ is Polish, there is $\xi \in (0, 1)_{\mathbb{R}}$ such that:
\begin{equation}\label{equation_1}
\text{ if } \bar{a}_\ell \in G^{n(\ell)} \text{ and } d(\bar{a}_{\ell}, \bar{g}^\ell_{\alpha(\ell)}), \text{ then } d(\sigma(\bar{a}^0_{0}, ..., \bar{a}^{\ell(*)-1}_{\ell(*)-1}), e) > \varepsilon/2.
\end{equation} 
Now, for $\ell < \ell(*)$, let $B_{\ell} = \{ \alpha \in A_{\ell} : d(\bar{g}^\ell_{\alpha}, \bar{g}^\ell_{\alpha_{\ell}}) < \xi \}$. Then $B_{\ell} \subseteq A_{\ell}$ and, as $\alpha(*) \leq \alpha_{\ell}$, clearly $|B_{\ell}| = \alpha_1$. Hence, by (\ref{equation_1}), $(B_{\ell} : \ell < \ell(*))$ is as required.
\end{proof}

\begin{theorem1}\label{theorem_half} No quantifier-free formula $\varphi(\bar{x}, \bar{y})$ in the language of group theory can define the $\aleph_1$-half graph in a Polish group $G$.
\end{theorem1}

	\begin{proof} Immediate from Theorem~\ref{stability_claim}.
\end{proof}

\section{The Space of Automorphism Groups of a Borel Complete Class}

%

	For an overview (and careful explanation) of the descriptive set theoretic notions occurring in this section cf. e.g. \cite[Chapter 11]{gao_invariant}. 

	\begin{notation}\label{notation_class_graphs1} We denote by $\mathbf{K}^{\aleph_0}_{\mathrm{gf}}$ the standard Borel space of graphs with domain~$\aleph_0$.
\end{notation}

	\begin{definition}\label{borel_complete} Let $\mathbf{K}$ be a Borel class of $L$-structures with domain $\aleph_0$. We say that $\mathbf{K}$ is {\em Borel complete} if there exists a Borel map $f: \mathbf{K}^{\aleph_0}_{\mathrm{gf}} \rightarrow \mathbf{K}$ such that for every $A, B \in \mathbf{K}_{0}$ we have $A \cong B$ if and only if $f(A) \cong f(B)$.
\end{definition}

\begin{notation}\label{top_iso} Let $G$ and $H$ be topological groups.
	\begin{enumerate}[(1)]
	\item We write $G \cong H$ to mean that $G$ and $H$ are isomorphic as abstract groups.
	\item We write $G \cong_t H$ to mean that $G$ and $H$ are isomorphic as topological groups.
	\end{enumerate}
\end{notation}

	\begin{notation}\label{notation_class_graphs} (As in Notation~\ref{notation_locally_finite}) Given a class $\mathbf{K}$ of $L$-structures, we let:
	$$Aut(\mathbf{K}) = \{ Aut(A) : A \in \mathbf{K} \}.$$
\end{notation}

	\begin{proposition2}\label{theorem2} Let $\mathbf{K}$ be a Borel class of $L$-structures with domain $\aleph_0$ such that for every $G \in Aut(\mathbf{K})$ we have that $|G| \leq \aleph_0$.  Then the isomorphism relation on $\mathbf{K}$ {\em is} Borel, and so in particular $\mathbf{K}$ is {\em not} Borel complete.
\end{proposition2}

	\begin{proof} We show that for any such class $\mathbf{K}$ the isomorphism relation $\cong$ on $\mathbf{K}$ is Borel. Notice that for $A, B \in \mathbf{K}$ we have that $A \cong B$ if and only if there are countably many $f \in S_{\infty} := \{ f: \omega \rightarrow \omega : f \text{ is a bijection} \}$ such that $f: A \cong B$. Thus, the relation $\cong$ on $\mathbf{K}$ is the projection of a Borel relation $R$:
	 $$(\mathbf{K} \times \mathbf{K}) \times S_\infty \supseteq R = \{ (A, B, f) : f: A \cong B\}$$ 
with countable sections $R_{(A, B)}$ (for $(A, B) \in \mathbf{K} \times \mathbf{K}$). Hence, by \cite[Lemma 18.12]{kech}, the relation $\cong$ on $\mathbf{K}$ is Borel, and so we are done.
\end{proof}

	We are interested in the following open problem:

\begin{problem}\label{problem_sec2} Let $\mathbf{K}$ be a Borel complete class. 
	\begin{enumerate}[(1)]
	\item Can $Aut(\mathbf{K})/\cong$ have size $1$? Can $Aut(\mathbf{K})/\cong_t$ have size $1$?
	\item Can $Aut(\mathbf{K})/\cong$ be finite? Can $Aut(\mathbf{K})/\cong_t$ be finite?
	\item Can $Aut(\mathbf{K})/\cong$ be countable? Can $Aut(\mathbf{K})/\cong_t$ be countable?
\end{enumerate}		
\end{problem}

\section{Groups of Automorphisms of Locally Finite Groups}

	\begin{definition}\label{def_pgroups}
	\begin{enumerate}[(1)]
	\item We denote by $P$ the set of prime numbers.
	\item For $p \in P$, we denote by $G^*_{(p, \infty)}$ the divisible abelian $p$-group of rank $1$.
	\item For $p \in P$ and $\ell < \omega$ we denote by $G^*_{(p, \ell)}$ the finite cyclic group of order~$p^\ell$.
	\item For $p \in P$, we let $S_p = \{ (p, \ell): \ell < \omega \}$ and $S^+_p = S_p \cup \{ (p, \infty) \}$.
	\item For $s \in S^+_p$ and $\lambda$ a cardinal, we let $G^*_{s, \lambda}$ be the direct sum of $\lambda$ copies of $G^*_s$.
	\item\label{def_J} For $p \in P$, we denote by $J_p$ the group of $p$-adic integers.
	\item We say that an abelian group $G$ is bounded if there exists $n < \omega$ such that for every $g \in G$ we have $n g = 0$.
	\item We say that $G$ is unbounded if it is not bounded.
	\item We say that $G$ is torsion if every element of $G$ has finite order.
	\end{enumerate}
\end{definition}

	\begin{fact}[\cite{fuchs}{[Theorem 17.2]}]\label{pre_fuchs_fact2}  Let $G$ be a bounded abelian group. Then $G$ is a direct sum of cyclic groups.
\end{fact}

	\begin{fact}\label{bounded_char} If an abelian $p$-group $G$ is bounded, then there exists $n < \omega$ such that:
	$$G = \bigoplus_{\ell < n} G^*_{(p, \ell), \lambda_\ell}.$$
\end{fact}

	\begin{proof} This is a consequence of Fact \ref{pre_fuchs_fact2}.
\end{proof}

	\begin{fact}[\cite{fuchs}{[Theorem 8.4]}]\label{p_part} Let $G$ be a torsion abelian group. Then:
	$$G = \bigoplus_{p \in P} G_p,$$
with $G_p$ a $p$-group, for every $p \in P$.
\end{fact}

	\begin{remark}[\cite{fuchs2}{[pg. 250]}]\label{auto_decomposition_fact} Let $G$ be an abelian group and suppose that $G = \bigoplus_{i \in I} G_i$, then $\bigoplus_{i \in I} Aut(G_i)$ can be embedded into $Aut(G)$.
\end{remark}

	\begin{fact}[{\cite[Theorem~115.1]{fuchs2}}]\label{J_fact} Let $G$ be an unbounded abelian $p$-group, then there exists an embedding $f: J_p \rightarrow Aut(G)$.
\end{fact}

	\begin{lemma}\label{element_finite_order_in_aut} Let $G \in \mathbf{K}^{\lambda}_{\mathrm{lf}}$ (cf. Definition~\ref{notation_locally_finite}). Then $Aut(G)$ has a non-trivial locally finite subgroup.
\end{lemma}

	\begin{proof} We distinguish three cases:
	\begin{enumerate}[(i)]
	\item $G$ is not abelian.
	\item $G$ is abelian and not bounded.
	\item $G$ is abelian and bounded. 
	\end{enumerate}
If (i), then $G/Cent(G) \in \mathbf{K}_{\mathrm{lf}}$ is non-trivial and it can be embedded into $Aut(G)$, and so we are done. If (ii), then we are done by Facts \ref{p_part} and \ref{J_fact}, and Remark \ref{auto_decomposition_fact}. Finally, if (iii), then by Facts \ref{bounded_char} and \ref{p_part}, there exists a direct summand $G_p$ of $G$ such that $G_p = \bigoplus_{\ell < n} G^*_{(p, \ell), \lambda_\ell}$ and for some $0 < \ell(*) < n < \omega$ we have $\lambda_{\ell(*)} \geq \aleph_0$. Let $G_{p^{\ell}}(\alpha)$ be the $\alpha$-th copy of $G^*_{(p, \ell), \lambda_\ell}$. If $p > 2$, then consider:
	$$\{ \pi \in Aut(G_p) : \pi \text{ maps } G_{p^{\ell}}(\alpha) \text{ onto itself, for every } \alpha < \lambda_{\ell} \text{ and } \ell < n \}.$$
	If $p = 2$, then consider:
	$$\{ \pi \in Aut(G_p) : \pi \text{ maps } G_{p^{\ell(*)}}(2\alpha) \oplus G_{p^{\ell(*)}}(2\alpha + 1) \text{ onto itself, for every } \alpha < \lambda_{\ell(*)}\}.$$
Hence, by Remark \ref{auto_decomposition_fact}, also (iii) is taken care of.
\end{proof}

	The following facts are folklore:
	
	\begin{fact}\label{auto_order}
	\begin{enumerate}[(1)]
	\item If $M$ is a linear order, then $Aut(M)$ has no element of finite order.
	\item For every infinite structure $M$ there exists a graph $\Gamma_M$ of the same cardinality of $M$ such that $Aut(\Gamma_M) \cong Aut(M)$.	
\end{enumerate}
\end{fact}

	\begin{theorem6} Let $\lambda \geq \aleph_0$, then:
	 $$Aut(\mathbf{K}^{\lambda}_{\mathrm{gf}}) \neq Aut(\mathbf{K}^{\lambda}_{\mathrm{lf}}).$$
\end{theorem6}

	\begin{proof} By Lemma \ref{element_finite_order_in_aut} and Fact \ref{auto_order}.
\end{proof}

	We devote the rest of the section to the proof of Proposition~\ref{lemma_for_countable} and Lemma~\ref{reduction_lemma}.

	\begin{notation}\label{notation_fixpoint} Given a group $G$ and $H \leq G$ we let:
	$$Aut_H(G) = \{ \pi \in Aut(G) : \pi \restriction H = id_H \}.$$
\end{notation}
	
	\begin{fact}\label{fact_for_conj_auto} If $G \in \mathbf{K}^{\lambda}_{\mathrm{lf}}$ and $|G/Cent(G)| < \lambda$, then there is $H \leq G$ such that $|H| < \lambda$ and $G = \langle H \cup Cent(G) \rangle_G$.
\end{fact}

	\begin{lemma}\label{lemma_for_conj_auto} If (A) then (B), where:
	\begin{enumerate}[(A)]
	\item \begin{enumerate}[(a)]
	\item $G \in \mathbf{K}^{\lambda}_{\mathrm{lf}}$;
	\item $H \leq G$ and $|H| < \lambda$;
	\item $G^* = Cent(G)$ and $G = \langle H \cup G^* \rangle_G$;
	\item $G^* = \bigoplus_{p \in P} G_{p}$ and $H_{p} = H \cap G_{p}$;
	\end{enumerate}
	\item \begin{enumerate}[(a)]
	\item if $\pi \in Aut_H(G)$, then $\pi(p) := \pi \restriction G_{p} \in Aut_{H_p}(G_{p})$;
	\item\label{mapping} the mapping $\pi \mapsto (\pi(p) : p \in P)$ from $Aut_H(G)$ into $\prod_{p \in P} Aut_{H_{p}}(G_{p})$ is an embedding;
	\item the embedding in (b) is onto.
	\end{enumerate}
	\end{enumerate}
\end{lemma}

	\begin{proof} The non-trivial part is item (B)(c). To this extent, let $\pi_p \in Aut_{H_p}(G_p)$, for $p \in P$. It suffices to find $\pi \in Aut_H(G)$ such that $\pi(p) = \pi_p$, for every $p \in P$. We define $\pi$ as follows. For $p_1 < \cdots < p_n \in P$ an initial segment of $P$ with the induced order, $y_{p_\ell} \in G_{p_\ell}$ and $y \in H$ we let:
	$$\pi(y y_{p_1} \cdots y_{p_{n}}) = y \pi_{p_1}(y_{p_1}) \cdots \pi_{p_n}(y_{p_{n}}).$$
Now, by (A), every $g \in G$ has at least one representation of the form $g = y y_{p_1} \cdots y_{p_{n}}$, and so, for every $g \in G$, $\pi(g)$ has at least one definition. We are then left to show that the choice of representation $g = y y_{p_1} \cdots y_{p_{n}}$ does not matter. To this extent, let $g \in G$ and suppose that:
$$y y_{p_1} \cdots y_{p_{m}} = g = y' y'_{p_1} \cdots y'_{p_{k}},$$
By adding occurrences of $e$ in the representations we can assume without loss of generality that:
$$y y_{p_1} \cdots y_{p_{n}} = g = y' y'_{p_1} \cdots y'_{p_{n}}.$$
Notice now that for $1 \leq \ell \leq n$ we have:
\begin{enumerate}[(a)]
	\item $y'_{p_{\ell}} \in y_{p_{\ell}} H_{p_{\ell}}$, say $y'_{p_{\ell}} = z_{p_\ell}y_{p_{\ell}}$ with $z_{p_\ell} \in H_{p_{\ell}}$;
	\item $y' y'_{p_1} \cdots y'_{p_{n}} = y' (z_{p_1}y_{p_{1}}) \cdots (z_{p_n}y_{p_{n}}) = (y'(z_{p_1} \cdots z_{p_n})) y_{p_1} \cdots y_{p_{n}}$;
	\item $y y_{p_1} \cdots y_{p_{n}} = (y'(z_{p_1} \cdots z_{p_n})) y_{p_1} \cdots y_{p_{n}}$;
	\item $y = y'(z_{p_1} \cdots z_{p_n})$.
	\end{enumerate}
Hence, we have:	
\[ \begin{array}{rcl}
\pi(y y_{p_1} \cdots y_{p_{n}}) & = & y \pi_{p_1}(y_{p_1}) \cdots \pi_{p_n}(y_{p_{n}}) \\
					  			& = & (y'(z_{p_1} \cdots z_{p_n})) \pi_{p_1}(y_{p_1}), ..., \pi_{p_n}(y_{p_{n}}) \\
					  			& = & y' z_{p_1}\pi_{p_1}(y_{p_1}) \cdots z_{p_n}\pi_{p_n}(y_{p_{n}}) \\
					  			& = & y' \pi_{p_1}(z_{p_1}y_{p_1}) \cdots \pi_{p_n}(z_{p_n}y_{p_{n}}) \\
					  			& = & y' \pi_{p_1}(y'_{p_1}) \cdots \pi_{p_n}(y'_{p_{n}}) \\
					  			& = & \pi(y' y'_{p_1} \cdots y'_{p_{n}}). 
\end{array}	\]
\end{proof}

	\begin{proposition8} If $G \in \mathbf{K}^{\aleph_0}_{\mathrm{lf}}$, then $Aut(G)$ has a locally finite infinite subgroup.
\end{proposition8}

\begin{proof} Let $G \in \mathbf{K}^{\aleph_0}_{\mathrm{lf}}$. If $G/Cent(G)$ is infinite, then we are done, since we can embed $G/Cent(G)$ into $Aut(G)$. So suppose that $G/Cent(G)$ is finite and let $G^*$, $G_p$ and $H_p$ be as in Lemma \ref{lemma_for_conj_auto}. If for some $p \in P$ we have that $G_p$ is infinite use Lemma~\ref{lemma_for_conj_auto} and Fact \ref{J_fact}, unless $G_p$ is bounded, in which case use Lemma~\ref{lemma_for_conj_auto} and Fact \ref{bounded_char}.
If for every $p \in P$ we have that $G_p$ is finite, then the set:
$$P^* = \{ p \in P : G_p \neq H_p \text{ and } [G_p : H_p] > 2 \}$$
is infinite, and so for every $p \in P^*$ we have that $Aut_{H_p}(G_p)$ is non-trivial. Hence, considering $\prod_{p \in P^*} Aut_{H_p}(G_p)$ and using Lemma~\ref{lemma_for_conj_auto} we are done.
\end{proof}

	\begin{lemma9} To prove Conjecture~\ref{conjecture_on_auto_intro} it suffices to prove Conjecture~\ref{conjecture_simplified}, where:
\end{lemma9}

	\begin{conjecture10} If $\lambda > \aleph_0$, $G \in \mathbf{K}^{\lambda}_{\mathrm{lf}}$ is an abelian $p$-group, $H \leq G$ and $|H| < \lambda$, then $Aut_H(G)$ has a locally finite subgroup of power~$\lambda$.
\end{conjecture10}

	\begin{proof}[Proof of Lemma~\ref{reduction_lemma}] By Fact \ref{fact_for_conj_auto} and Lemma \ref{lemma_for_conj_auto}.
\end{proof}

\end{document}